\documentclass[a4paper,onecolumn,10pt,accepted=2019-11-01, volume=1, issue=3, shorttitle=papers/compositionality-1-3]{compositionalityarticle}
\pdfoutput=1
\usepackage[T1]{fontenc}
\usepackage[utf8]{inputenc}
\usepackage{hyperref}
\usepackage{authblk}
\usepackage[numbers,sort&compress]{natbib}
\usepackage{amsmath}
\usepackage{amsthm}
\usepackage{amsfonts}
\usepackage{amssymb}
\usepackage[all]{xy}
\usepackage{url}

\usepackage{graphicx}
\DeclareGraphicsExtensions{.pdf,.jpeg,.png}
\DeclareMathSymbol{\rightrightarrows}  {\mathrel}{AMSa}{"13}

\def\Shv{\mathbf{Shv}}

\def\Pre{\mathbf{Pre}}

\def\Mon{\mathbf{Mon}}
\def\Fuz{\mathbf{Fuzz}}
\def\Sub{\mathbf{Sub}}
\def\Im{\operatorname{Im}}

\catcode`\@=11
\def\varholim@#1#2{\mathop{\vtop{\ialign{##\crcr
 \hfil$#1\m@th\operator@font holim$\hfil\crcr
 \noalign{\nointerlineskip\kern\ex@}#2#1\crcr
 \noalign{\nointerlineskip\kern-\ex@}\crcr}}}}
\def\hocolim{\mathpalette\varholim@\rightarrowfill@} 
\def\hoinvlim{\mathpalette\varholim@\leftarrowfill@}
\catcode`\@=\active

\newtheorem{theorem}{Theorem}%[section]
\newtheorem{lemma}[theorem]{Lemma}
\newtheorem{corollary}[theorem]{Corollary}

\theoremstyle{definition}

\newtheorem{example}[theorem]{Example}

\newtheorem{remark}[theorem]{Remark}

\begin{document}

\title{Fuzzy sets and presheaves}
\author{John F.\ Jardine}
\affiliation{Department of Mathematics, University of Western Ontario,
  London, Ontario, Canada}
\email{jardine@uwo.ca}
\thanks{Supported by NSERC.}
\maketitle

\begin{abstract}
\noindent
This paper presents a presheaf theoretic approach to the construction of fuzzy sets, which builds on Barr's description of fuzzy sets as sheaves of monomorphisms on a locale. Presheaves are used to give explicit descriptions of limit and colimit descriptions in fuzzy sets on an interval. The Boolean localization construction for sheaves on a locale specializes to a theory of stalks for sheaves and presheaves on an interval.

The system $V_{\ast}(X)$ of Vietoris-Rips complexes for a data set $X$ is both a simplicial fuzzy set and a simplicial sheaf in this general framework. This example is explicitly discussed through a series of examples.
\end{abstract}

\section{Introduction}

Fuzzy sets were originally defined to be functions
$
\psi: X \to [0,1]
$
that take values in the unit interval \cite{DuPr}.

Michael Barr changed the game in his paper \cite{BarrF}: he replaced the unit interval by a more general well-behaved poset, a locale $L$, and redefined fuzzy sets to be functions
\begin{equation*}
\psi: X \to L.
\end{equation*}
These functions are the fuzzy sets over $L$, and form a category $\Fuz(L)$ that is described in various ways below.

Locales are complete lattices in which finite intersections distribute over all unions. The unit interval $[0,1]$ qualifies, but so does the poset of open subsets of a topological space, and locales are models for spaces in this sense.

Every locale $L$ has a Grothendieck topology, with coverings defined by joins, and so one is entitled to a presheaf category $\Pre(L)$ and a sheaf category $\Shv(L)$ for such objects. Presheaves on $L$ are contravariant set-valued functors on $L$, and sheaves are presheaves that satisfy a patching condition with respect to the Grothendieck topology on $L$.

Barr showed that, starting with a fuzzy set $\psi: X \to L$ over $L$, one can pull back over subobjects to define a sheaf $T(\psi)$, which is a sheaf for which all restriction maps are injections. 
Technically, one needs to adjoin a new zero object to $L$ to make this work, giving a new (but not so different) locale $L_{+}$. The resulting object $T(\psi)$ is a sheaf of monomorphisms on $L_{+}$ in the sense that all non-trivial restriction maps are injective functions.

Write $\Mon(L_{+})$ for the category of sheaves of monomorphisms on $L_{+}$. Barr showed \cite{BarrF} that his functor
\begin{equation*}
T: \Fuz(L) \to \Mon(L_{+})
\end{equation*}
is part of a categorical equivalence. This result appears as Theorem \ref{th 12} in the first section of this paper.

The inverse functor for $T$ on a sheaf $F$ is constructed by taking the {\it generic fibre} $F(i)$, and constructing a function $\psi_{F}: F(i) \to L$. The set $F(i)$ is the set of sections corresponding to the initial object $i$ of $L$.  Given an element $x \in F(i)$, there is a maximum $s_{x} \in L$ such that $x$ is in the image of the monomorphism $F(s_{x}) \to F(i)$, and one defines $\psi_{F}(x)$ to be this element $s_{x}$.
\medskip

The first section of this paper is largely expository and self contained. We set notation and introduce the main examples in modern terms, and present a proof of the the Barr result (Theorem \ref{th 12}) that is expressed in this newer language.

Examples \ref{ex 6} and Example \ref{ex 10} show, respectively, that the Vietoris-Rips filtration corresponding to a data set has the structure of a simplicial fuzzy set and (through the Barr theorem) a simplicial sheaf.

To proceed with applications, for example if one wants to sheafify peristent homology theory or clustering and use fuzzy sets to do it, or to say anything about the homotopy types of simplicial objects, it is helpful to have more explicit information about how fuzzy sets are constructed. 
One needs, in particular, straightforward descriptions of
basic constructions such as limits, colimits and stalks in the fuzzy set category, or rather in the associated category of sheaves of monomorphisms.
The difficulties, such as they are, arise from the fact that the category $\Mon(L_{+})$ of sheaves of monomorphisms is not quite a sheaf category, and constructing the fuzzy set $\psi_{F}: F(i) \to L$ from a sheaf $F$ can be a bit interesting.

These issues are dealt with in Sections 2 and 3 of this paper. There is a perfectly good category  $\Mon_{p}(L_{+})$ of presheaves of monomorphisms, and it turns out that if $L$ is sufficiently well behaved (as is the unit interval $[0,1]$), then the associated sheaf functor is easily described and preserves presheaves of monomorphisms. The upshot is that one can make constructions on the presheaf category, as a geometer or topologist would, and then sheafify. 

Limits are formed as in the ambient sheaf category, meaning sectionwise, but colimits are more involved. The inclusion $\Mon(L_{+}) \subset \Shv(L_{+})$ of sheaves of monomorphisms in all sheaves has a left adjoint $F \mapsto \Im(F)$, called the {\it image functor}, which is defined by taking images of sets of sections in the generic fibre --- see Lemma \ref{lem 26}. This observation allows one to define colimits of diagrams $A(i)$ in $\Mon(L_{+})$: take the presheaf theoretic colimit $\varinjlim_{i}\ A(i)$, and then apply the image functor (sheafified) to get the colimit in $\Mon(L_{+})$. 

The image functor and colimit constructions are described in Section 2. That section also contains the formal definitions and properties around presheaves of monomorphisms.

One has the nicest form of the associated sheaf functor for presheaves on a locale $L$ when one assumes that $L$ is an interval (Lemma \ref{lem 22}). The interval assumption on $L$ is consistent with the classical theory of fuzzy sets and with the intended applications in Topological Data Analysis. 

The general theory of Boolean localization for sheaves and presheaves on a locale $L$ is relatively straightforward to describe, and is the starting point for Section 3 of this paper.

Every locale $L$ has a standard imbedding $\omega: L \to B$ into a complete Boolean algebra, by a rather transparent construction that is displayed here (see also \cite{MM}, \cite{J21}, \cite{LocHom}, for example). This is the easier part of the general Boolean localization construction --- the more interesting bit is the construction of the Diaconescu cover, which faithfully imbeds a Grothendieck topos in the topos of sheaves on a locale.

If the locale $L$ is an interval, then the corresponding Boolean algebra $B$ is the set of subsets of some set (Example \ref{ex 31}), so that the sheaf category for $L$ has enough points, and therefore has a theory of stalks. The same is true for finite products of intervals (Example \ref{ex 33}).

The description of stalks for sheaves of monomorphisms on an interval that arises from the general Boolean localization construction is fairly simple, and can be used as a starting point for a result (Lemma \ref{lem 32}) that expresses what stalks are supposed to do in this instance, while avoiding the abstract Boolean localization machinery.

The final part of Section 3 consists of a description of stalks for presheaves on an interval $L$. The construction of stalks for presheaves is a left Kan extension construction, by analogy with stalks for presheaves on a topological space or on the \'etale site of a scheme.
\medskip

This paper was written to clear the air about the sheaf theoretic properties of fuzzy sets, and to set the stage for potential
applications of the local homotopy theory of simplicial presheaves. See the Healy-McInnes paper \cite{HMc-2018} for a recent discussion of applications of simplicial fuzzy sets in topological data analysis.

We see in Example \ref{ex 10} that the system of Vietoris-Rips complexes $s \mapsto V_{s}(X)$ that is associated to a data set $X \subset \mathbb{R}^{n}$ does form a simplicial fuzzy set, or a simplicial sheaf (of monomorphisms) on  the locale $[0,R]^{op}$, where $R$ is larger than all distances between points of $X$. But we also see in Example \ref{ex 34} that this simplicial sheaf on $[0,R]^{op}$ has a rather awkward collection of stalks, which includes the full data set $X$ sitting as a discrete simplicial set in the generic fibre. It follows, for example, that an inclusion $X \subset Y \subset \mathbb{R}^{n}$ induces a stalkwise weak equivalence of simplicial sheaves if and only if $X=Y$.

This is quite like the situation that was encountered in the first attempt to give a sheaf theoretic context for topological data analysis \cite{Ober}. The earlier paper uses a different topology on the underlying space of parameters, but produces essentially the same stalks and thus has the same problem with local weak equivalences that are too tightly defined to be useful. 

A better option might be to use the metric space $D(Z)$ (Hausdorff metric) of finite subsets of a fixed metric space $Z$ as the base topological object. This is the setting for modern stability results \cite{CM-2010B}, \cite{BlumLes}, which roughly assert that if two data sets $X,Y$ in $Z$ are close in the sense that there is a specific bound $r$ on their Hausdorff distance, then the associated persistence invariants (homotopy types, hence clusters and persistent homology) have an interleaving distance with a specific upper bound, usually $2r$. In particular, if $X$ and $Y$ are close in the Hausdorff metric, then the corresponding systems $\{V_{s}(X)\}$ and $\{V_{s}(Y)\}$ are tightly interleaved as homotopy types. This is a local principle, and the present aim is to globalize it. The challenge is to interpret homotopy interleaving in terms of a local homotopy theoretic structure associated to the space $D(Z)$.

%\vfill\eject

%\tableofcontents

\section{Fuzzy sets and sheaves}

This section gives a general introduction to Barr's theory of fuzzy sets over a locale, in modern language and with multiple examples. We prove Barr's Theorem that the category of fuzzy sets over a locale $L$ is equivalent to the category of sheaves of monomorphisms on a slightly augmented version $L_{+}$ of $L$ --- this is Theorem \ref{th 12} below.

Every closed interval $[0,R]$ is a locale, and we identify the Vietoris-Rips system $V(X)$ with a simplicial fuzzy set over $[0,R]^{op}$ ($R$ sufficiently large) in Example \ref{ex 10}, or equivalently with a simplicial sheaf of monomorphisms, via Barr's Theorem.

The section finishes with a discussion of completeness properties for the category of sheaves of monomorphisms on $L_{+}$ and hence of fuzzy sets over $L$.
\medskip

A {\it frame} $L$ is a complete lattice in which finite meets distribute over all joins. Examples include the poset $op\vert_{X}$ of open subsets of a topological space $X$.

According to this definition, $L$ has a terminal object
$1$ (empty meet) and an initial object $0$ (empty join) --- see \cite[p.471]{MM}.

The completeness assumption means that every set of elements $a_{i} \in L$ has a least upper bound $\vee a_{i}$. The set $a_{i}$ also has a greatest lower bound $\wedge a_{i}$, which is the least upper bound $\vee_{x \leq a_{i}}\ x$ of the elements $x$ which are smaller than all $a_{i}$.

A morphism $L_{1} \to L_{2}$ of frames is a poset morphism which preserves meets and joins, and hence preserves initial and terminal objects. The category of locales is the opposite of the frame category, and one uses the terms ``frame'' and ``locale'' interchangeably. 

\begin{example}
    The interval $[0,1]$ of real numbers $s$ with $0 \leq s \leq 1$, with the standard ordering, is a locale. 

    The poset $[0,1]$ imbeds in $op\vert_{[0,1]}$ by the assignment $a \mapsto [0,a)$. Then $s \wedge t$ corresponds to the interval
\begin{equation*}
      [0,s) \cap [0,t) =[0, s\wedge t),
\end{equation*}
so that $s \wedge t = \min \{s,t\}$. Similarly,
\begin{equation*}
  \cup_{i}\ [0,s_{i}) = [0,\vee_{i}\ s_{i}),
\end{equation*}
where $\vee_{i}\ s_{i}$ is the least upper bound of the numbers $s_{i}$. 
\end{example}

\begin{example}
  The interval $[0,1]$ with the opposite ordering $[0,1]^{op}$ is also a locale. Here, $s \leq t$ in $[0,1]^{op}$ if and only if $t \leq s$ in $[0,1]$.

  In this case, $[0,1]^{op}$ imbeds in $op\vert_{[0,1]}$ by the assignment $s \mapsto (s,1]$. Then $\vee_{i}\ s_{i}$ is the greatest lower bound of the $s_{i}$ and $s \wedge t = \max \{s,t\}$.
\end{example}

    Observe that the posets $[0,1]$ and $[0,1]^{op}$ both have infinite meets, given by greatest lower bound and least upper bound, respectively.

    \begin{example}
    The closed interval $[a,b]$ and its opposite $[a,b]^{op}$ are locales, and there are linear scaling isomorphisms $[a,b] \cong [0,1]$ and $[a,b]^{op} \cong [0,1]^{op}$.
    \end{example}

    \begin{example}
      Suppose that $L_{1}, \dots ,L_{k}$ are locales. Then the product poset
      \begin{equation*}
        L_{1} \times \dots \times L_{k}
      \end{equation*}
      is also a locale.
    \end{example}

Suppose that $L$ is a locale. Following  \cite{BarrF}, a  function $\psi: X \to L$ is a {\it fuzzy set} over $L$. These are the objects of a category $\Fuz(L)$, called the category of fuzzy sets over $L$.

Suppose that $\phi: Y \to L$ is another such function. A morphism $(f,h): \psi \to \phi$ of $\Fuz(L)$ consists of a function $f: X \to Y$ and a relation $h: \psi \leq \phi \cdot f$ of functions taking values in the poset $L$. The existence of the relation $h$ means precisely that $\psi(x) \leq \phi(f(x))$ in the poset $L$ for all $x \in X$.

There is a poset $L^{X}$ whose objects are the functions $\psi: X \to L$. If $\gamma: X \to L$ is another such function, then there is a relation $\psi \leq \gamma$ if $\psi(x) \leq \gamma(x)$ for all $x \in X$. Every function $f: X \to Y$ determines a restriction functor $L^{Y} \to L^{X}$, so that there is a contravariant functor $\mathbf{Set} \to \mathbf{cat}$ which is defined by associating the poset $L^{X}$ to the set $X$.

Following Quillen (see \cite{GJ2}, for example), a {\it homotopy} $h: \psi \to \phi\cdot f$ is a natural transformation between functors $L^{Y} \to L^{X}$, which in the case at hand is given by the relations $\psi(x) \leq \phi(f(x)), x \in X$.

From this point of view, a morphism of $\Fuz(L)$ is a morphism
\begin{equation*}
  (f,h) : \psi \to \phi
\end{equation*}
  in the Grothendieck construction associated to the diagram of restriction functors, and the fuzzy set category $\Fuz(L)$ is that Grothendieck construction.

\begin{example}
All commutative diagrams
\begin{equation*}
\xymatrix@C=10pt{
X \ar[rr]^{g} \ar[dr] && Y \ar[dl] \\
& L
}
\end{equation*}
correspond to morphisms $(g,1)$ of $\Fuz(L)$ with identity homotopies in $L^{X}$, but the full collection of fuzzy set morphisms $X \to Y$ is larger --- these are the homotopy commutative diagrams.
\end{example}

\begin{example}\label{ex 6}
  Suppose that a finite set $X \subset \mathbb{R}^{n}$ is a data set, and suppose that $X \xrightarrow{\cong} \mathbf{N}$ is a listing of the members of $X$, where $\mathbf{N} = \{0,1, \dots ,N\}$. Choose a number $R$ such that $d(x,y) < R$ for all $x,y \in X$.

  Here, $d(x,y)$ is the distance between the points $x$ and $y$ in $\mathbb{R}^{n}$.

Let $\sigma$ be an ordered set of points $\sigma = \{x_{0},x_{1}, \dots ,x_{k}\}$ in $X$. Write
\begin{equation*}
  \phi(\sigma) = \max_{i,j}\ d(x_{i},x_{j})
\end{equation*}
Suppose that $\theta: \mathbf{r} \to \mathbf{k}$ is an ordinal number map.
then $\phi(\theta^{\ast}(\sigma)) \leq \phi(\sigma)$, with equality if $\theta$ is surjective, or if $\theta^{\ast}(\sigma)$ is a degeneracy of $\sigma$. Further, $\phi(\sigma)=0$ if and only if $\sigma$ is a degeneracy of a vertex.

The assignment $\sigma \mapsto \phi(\sigma)$ defines a function
\begin{equation*}
  \phi: \Delta^{N}_{k} \to [0,R].
\end{equation*}
on the set $\Delta^{N}_{k}$ of $k$-simplices of the simplicial set $\Delta^{N}$.

If $\theta: \mathbf{r} \to \mathbf{k}$ is an ordinal number map, then the relation  $\phi(\theta^{\ast}(\sigma)) \leq \phi(\sigma)$ defines a homotopy commutative diagram
\begin{equation}\label{eq 1}
  \xymatrix@C=10pt{
    \Delta^{N}_{k} \ar[rr]^{\theta^{\ast}} \ar[dr]_{\phi}  && \Delta^{N}_{r} \ar[dl]^{\phi} \\
        &  [0,R]^{op} 
            }
\end{equation}
or equivalently a morphism of fuzzy sets with values in the locale $[0,R]^{op}$.

The ordering $X \cong \mathbf{N}$ on the elements of the data set $X$ and the ambient distance function on $\mathbb{R}^{n}$ combine to give the simplicial set $\Delta^{N}$ the structure of a simplicial fuzzy set $\phi: \Delta^{N} \to [0,R]^{op}$, with coefficients in the locale $L = [0,R]^{op}$.
%\smallskip

We shall write $\Delta^{X}$ for $\Delta^{N}$ henceforth to suppress notational dependence on the ordering $N \cong X$. The homotopy types of the spaces $V_{s}(X)$ are independent of the ordering on $X$ in any case. The simplicial fuzzy set associated to a data set $X$ therefore has the form $\phi: \Delta^{X} \to [0,R]^{op}$.
\end{example}

A {\it simplicial fuzzy set} $Z$ is a simplicial object in $\Fuz(L)$, meaning a contravariant functor $Z: \mathbf{\Delta}^{op} \to \Fuz(L)$ on the category of finite ordinal numbers. This usage is standard: a {\it simplicial object} in a category $\mathcal{A}$ is a functor $\mathbf{\Delta}^{op} \to \mathcal{A}$. See \cite{GJ2}.

The first appearance of simplicial fuzzy sets in the literature may be in Spivak's preprint \cite{fuzzy-Spivak} of 2009, where these objects are called fuzzy simplicial sets. The explicit interpretation of the Vietoris-Rips filtration as a simplicial fuzzy set that is presented in Example \ref{ex 6} seems to be new, but see the Healy-McInnes paper \cite{HMc-2018}.
\medskip  

Suppose that $L$ is a locale. Then $L_{+} = L \sqcup \{0 \}$ is also a locale, where $0$ is a new initial element. 

\begin{remark}
If $L=[0,R]^{op}$ then the object $R$ is no longer initial in $L_{+}$. I normally write $i$ for the number $R$ (the original initial object of $[0,R]^{op}$) to distinguish this element from the initial object $0$ of $L_{+}$. Clearly, $0 < i$ in $L_{+}$.
\end{remark}

Any locale $L$ has a Grothendieck topology, for which the {\it covering families} of $a \in L$ are sets of objects $b_{i} \leq a$ such that $\vee_{i} b_{i} = a$. This relation is equivalent to the assertion that $a$ is the least upper bound in $L$ for all elements $b_{i}$.

Given a family of elements $b_{i} \leq a$, the associated {\it sieve} $R$ is the set of all elements $s$ such that $s \leq b_{i}$ for some $i$. The sieve $R$ is {\it covering} if $\{ b_{i} \}$ is a covering family.

Equivalently, an arbitrary sieve $R$, i.e. a subset of the collection of elements $s \leq a$ which is closed under taking subobjects, is {\it covering} if $\vee_{s \in R}\ s = a$.

Since $L$ has a Grothendieck topology, it has associated categories $\Pre(L)$ and $\Shv(L)$ of presheaves and sheaves on $L$, respectively.

A {\it presheaf} is a functor $F: L^{op} \to \mathbf{Set}$, and a morphism of presheaves is a natural transformation. 

One says that the presheaf $F$ is a {\it sheaf}
if the map
\begin{equation*}
F(a) \to \varprojlim_{b \in R}\ F(b)
\end{equation*}
is an isomorphism for all covering sieves $R$ of all objects $a$. This is equivalent to requiring that the diagram
\begin{equation*}
F(a) \to \prod_{i}\ F(b_{i}) \rightrightarrows \prod_{i,j}\ F(b_{i} \wedge b_{j})
\end{equation*}
is an equalizer for all covering families $\{ b_{i} \}$ of all objects $a$. In other words, $F(a)$ should be recovered from the values of $F(b_{i})$ by patching, for all coverings $\{ b_{i} \}$ of $a$.

\begin{remark}
  If $0$ is an initial object of $L$ and $F$ is a sheaf on $L$, then $F(0)$ must be the one-point set. One writes $F(0) = \ast$ to express this.

  In effect, the empty sieve $\emptyset \subset \hom(\ ,0)$ is covering, because $0$ is an empty join. It follows that there is an isomorphism
  \begin{equation*}
    F(0) \cong \hom(\hom(\ ,0),F) \cong \hom(\emptyset,F) = \ast
  \end{equation*}
  for any sheaf $F$.
  Compare with \cite[p.35]{LocHom}.
\end{remark}

We are therefore entitled to categories $\Pre(L_{+})$ and $\Shv(L_{+})$ of presheaves and sheaves, respectively for the locale $L_{+}$, and these are the examples that we will focus on.

Write $\mathbf{Mon}(L_{+})$ for the full subcategory of the sheaf category $\Shv(L_{+})$, whose objects are the sheaves $F$ such that all restriction maps $F(b) \to F(a)$ associated to relations $a \leq b$ in $L$ are monomorphisms.
The requirement that the relation $a \leq b$ is in $L$ is important, because $F(0) = \ast$.

Barr constructs a functor \cite{BarrF}
\begin{equation*}
T: \Fuz(L) \to \Mon(L_{+})
\end{equation*}
which defines an equivalence of categories. The existence of this equivalence of categories is the main result of \cite{BarrF} , and it appears as Theorem \ref{th 12} below.

Explicitly, define
\begin{equation*}
L_{\geq a} = \{\ x \in L\ \vert\ x \geq a\ \}.
\end{equation*}
If $\psi: X \to L$ is a member of $\Fuz(L)$, define a presheaf $T(\psi)$ by
\begin{equation*}
T(\psi)(a) = \psi^{-1}(L_{\geq a}).
\end{equation*}
for $a \in L$, and set $T(\psi)(0) = \ast$.
Then the assignment $a \mapsto T(\psi)(a)$ defines a presheaf on $L_{+}$ such that every relation $s \leq t$ induces a monomorphism $T(\psi)(t) \to T(\psi)(s)$. The presheaf $T(\psi)$ is a sheaf because $x \in T(\psi)(a)$ if and only if $x \in T(\psi(b_{i}))$ for any covering family $\{ b_{i} \}$ of $a$.

If $(f,h): \psi \to \phi$ is a morphism of fuzzy sets (as above), then the relations $\psi(x) \leq \phi(f(x))$ in $L$ (i.e. the homotopy $h$) imply that if $\psi(x) \in L_{\geq a}$ then $\phi(f(x)) \in L_{\geq a}$, and so the function $f$ restricts to functions
\begin{equation*}
  f: \psi^{-1}L_{\geq a} \to \phi^{-1}L_{\geq a}
\end{equation*}
  that are natural in $a$, so that we have a sheaf homomorphism
\begin{equation*}
  f_{\ast}: T(\psi) \to T(\phi).
\end{equation*}

\begin{remark}
$T(X)$ is sometimes called the {\it level cut} description of the fuzzy set $\psi: X \to L$ --- see \cite{DuPr}. 
\end{remark}

\begin{example}\label{ex 10}
  Suppose that the finite set $X \subset \mathbb{R}^{n}$ is a data set, with ordering $X \xrightarrow{\cong} \mathbb{N}$ as in Example \ref{ex 6}. Again, choose $R > d(x,y)$ for all pairs of points $x,y \in X$.

  Recall that the simplicial fuzzy set $\phi: \Delta^{X}_{k} \to [0,R]^{op}$ is defined for a simplex $\sigma = \{ x_{0},x_{1} , \dots ,x_{k} \}$ by
    \begin{equation*}
      \phi(\sigma) = \max_{i,j}\ d(x_{i},x_{j}).
    \end{equation*}
    Then, for $s \in [0,R]$,
      \begin{equation*}
        T(\phi)_{s} = \phi^{-1}[0,s],
      \end{equation*}
      which is the set of $k$-simplices $\sigma = \{ x_{0}, \dots ,x_{k} \}$ of $\Delta^{X}$ such that $d(x_{i},x_{j}) \leq s$. It follows that
$
        T(\phi)_{s} = V_{s}(X)_{k}
$
      is the set of $k$-simplices of the Vietoris-Rips complex $V_{s}(X)$ for the data set $X$.

      The Vietoris-Rips complex functor $s \mapsto V_{s}(X)$ is the simplicial sheaf that Barr's construction associates to the simplicial fuzzy set $\phi: \Delta^{X} \to [0,R]^{op}$.
\end{example}

There is an isomorphism 
\begin{equation}\label{eq 2}
\varinjlim_{a \in L}\ F(a) \xrightarrow{\cong} F(i),
\end{equation}
for any sheaf $F \in \Mon(L_{+})$, since $i$ is initial in $L$. This colimit is filtered, and the canonical maps $F(a) \to F(i)$ are monomorphisms. 

The set of sections $F(i)$ is the {\it generic fibre} of the object $F$.

\begin{lemma}\label{lem 11}
Suppose that $F$ is a sheaf of monomorphisms on $L_{+}$ and that $x \in F(i)$. Then there is a unique maximum element $s_{x}$ such that $x \in F(s_{x})$.
\end{lemma}

\begin{proof}
Consider all $c$ in $L$ such that $x \in F(c)$, and let 
\begin{equation*}
s_{x} = \vee_{x \in F(c)}\ c.
\end{equation*}
Then $s_{x}$ is covered by the elements $c$, and so $x \in F(s_{x})$.
\end{proof}

Suppose again that $F \in \Mon(L_{+})$. 
By Lemma \ref{lem 11}, for each $x \in F(i)$, there is a unique maximum $s_{x}$ such that $x \in F(s_{x})$. Define $\psi(F): F(i) \to L$ by setting $\psi_{F}(x) = s_{x}$. Then we have a function
\begin{equation*}
\psi(F): F(i) \to L, 
\end{equation*}
which is a fuzzy set. 

To put it a slightly different way, the fuzzy set $\psi_{F}: F(i) \to L$ is defined by 
\begin{equation*}
\psi_{F}(x) = sup\ \{ b\ \vert\ x \in F(b)\ \}
\end{equation*}
for $x \in F(i)$, and $F \in \Mon(L_{+})$.

\begin{theorem}[Barr]\label{th 12}
  The assignments $\psi$ and $\phi$ define an equivalence of categories
  \begin{equation*}
    \psi: \Mon(L_{+}) \overset{\simeq}{\leftrightarrows} \Fuz(L): T
  \end{equation*}
\end{theorem}

\begin{proof}
  Suppose that $F$ is a sheaf of monomorphisms, and that $b \in L$, and let $\psi_{F}: F(i) \to L$ be the corresponding fuzzy set. Then
\begin{equation*}
  F(b) = \psi(F)^{-1}(L_{\geq b})
\end{equation*}
as subsets of $F(i)$, so that there is a natural sheaf isomorphism
  \begin{equation*}
    F \xrightarrow{\cong} T(\psi(F)).
    \end{equation*}
    
Suppose that $f: X \to L$ is a fuzzy set. Then $\psi(f)(i) = f^{-1}(L) = X$ and $x \in f^{-1}(L_{\geq f(x)})$. If $x \in f^{-1}(L_{\geq b})$ for some $b$, then $f(x) \geq b$. It follows that $f(x) = \psi(T(f))(x)$.
  \end{proof}

\begin{example}
The representable functor $\hom(\ ,s)$ on $L_{+}$ has the form
\begin{equation*}
\hom(\ ,s)(t) = 
\begin{cases}
\ast & \text{if $t \leq s$, and} \\
\emptyset & \text{otherwise}.
\end{cases}
\end{equation*}
Here, $\ast$ is the one-point set.

This presheaf is a sheaf, so the topology on $L_{+}$ is sub-canonical.
The sheaf $\hom(\ ,s)$ is a sheaf of monomorphisms. The corresponding fuzzy set, for $s \in L$, is the function $s:\ast \to L$ which picks out the element $s \in L$.

The constant presheaf $\ast$ is defined to be a one-point set $\ast(a)$ for all $a \in L_{+}$, with identity maps associated to all relations $a \leq b$. This presheaf is a sheaf, and is a member of $\Mon(L_{+})$. This sheaf is represented by the terminal object $t \in L$, and so the corresponding fuzzy set is the function $t: \ast \to L$.

If $s \leq t$ in $L$, then the induced sheaf map $\hom(\ ,s) \to \hom(\ ,t)$ corresponds to the fuzzy set map from $s: \ast \to L$ to $t: \ast \to L$ which is given by the identity function on $\ast$ and the relation $s \leq t$.
\end{example}

\begin{example}
Suppose that $K$ is a simplicial set. The simplicial presheaf $L_{s}(K)$ that is defined by
\begin{equation*}
  L_{s}(k) = \hom(\ ,s) \times K
\end{equation*}
is a simplicial sheaf of monomorphisms, and therefore represents a simplicial fuzzy set. There is a natural isomorphism
\begin{equation*}
  \hom(L_{s}(K),X) \cong \hom(K,X(s))
\end{equation*}
for all simplicial sheaves (or presheaves) $X$ on $L_{+}$. See also
\cite[Sec. 2.3]{LocHom}.
\end{example}

\begin{lemma}\label{lem 15}
  The category $\mathbf{Mon}(L_{+})$ is complete. Limits are formed in the ambient sheaf category $\mathbf{Shv}(L_{+})$.
\end{lemma}

\begin{proof}
  This result follows from the fact that an inverse limit of monomorphisms is a monomorphism.
  \end{proof}

\begin{example}\label{ex 16}
Form the pullback diagram
\begin{equation*}
\xymatrix{
Z \ar[r] \ar[d] & F \ar[d]^{q} \\
E \ar[r]_{p} & X
}
\end{equation*}
of sheaves on $L_{+}$, with $E,F,X$ all in $\Mon(L_{+})$. 

Take 
\begin{equation*}
(x,y) \in Z(i) = E(i) \times_{X(i)} F(i)
\end{equation*}
and suppose that $(x,y) \in Z(a)$.

Then $x \in E(a)$ and $y \in F(a)$ so that $a \leq \psi_{E}(x)$ and $a \leq \psi_{F}(y)$. It follows that $a \leq \psi_{E}(x) \wedge \psi_{F}(y)$.

On the other hand, if $b \leq  \psi_{E}(x) \wedge \psi_{F}(y)$, then there is a $v \in E(b)$ which restricts to $x$ and a $u \in F(b)$ which restricts to $y$. Also, $p(v)$ and $q(u)$ in $X(b)$ restrict to the same element of $X(i)$, so that $p(v) = q(u)$, in $X(b)$ and $(u,v) \in Z(b)$.

It follows that
\begin{equation*}
\psi_{Z}((x,y)) = \psi_{E}(x) \wedge \psi_{F}(y)
\end{equation*}
for all $(x,y) \in Z(i)$.

Another way of saying this is to assert that $\psi_{Z}((x,y))$ is the greatest lower bound of $\psi_{E}(x)$ and $\psi_{E}(y)$.
\end{example}

\begin{example}
Suppose that $X: J \to \Mon(L_{+})$ is a small diagram.  For a fixed object $a \in L_{+}$, the $a$-sections of $Z=\varprojlim_{j}\ X(j)$ are the $J$-compatible families $\{ x_{j} \}$ of elements in the various sets $X(j)(a)$.

One can use the methods of the pullback case in Example \ref{ex 16} to show that
$\psi_{Z}(\{ x_{j} \})$ is the greatest lower bound in $L_{+}$ of the elements $\psi_{X(j)}(x_{j})$.
\end{example}

\section{Presheaves of monomorphisms}

Suppose that $L$ is a locale. This section introduces the theory of presheaves of monomorphisms on the augmented locale $L_{+}$.

If $L$ is an interval in a suitable sense, then all presheaves of monomorphisms $F$ on $L_{+}$ are separated, and the category of coverings for $s \in L$ has a particularly simple form (Lemma \ref{lem 18}). The associated sheaf for $F$ is a sheaf of monomorphisms in this case (Lemma \ref{lem 22}).

Any presheaf on $L_{+}$ has an associated presheaf of monomorphisms $\Im(F)$, which is defined by taking images in a generic fibre (Lemma \ref{lem 26}). This result immediately yields a colimit construction for sheaves of monomorphisms: take the usual presheaf-theoretic colimit, apply the image functor, and sheafify (Lemma \ref{lem 28}). Examples of colimit constructions for fuzzy sets are discussed at the end of the section.
\medskip

We now consider presheaves $F: (L_{+})^{op} \to \mathbf{Set}$ such that $F(0)=\ast$ and all morphisms $a \leq b$ of $L$ induce monomorphisms $F(b) \to F(a)$. Such a presheaf is called a {\it presheaf of monomorphisms}.

Write $\Mon_{p}(L_{+})$ for the category of presheaves of this form. 
\medskip

Most of the results of this section depend on the assumption that
the locale $L$  is an {\it interval} in the sense that
\begin{itemize}
\item[1)] $L$ has a total ordering, and
\item[2)] the ordering is {\it dense}, meaning that if $a < b$ in $L$, there is an $s \in L$ such that $a < s < b$.
\end{itemize}

The locales of immediate practical interest, such a closed interval $[c,d] \subset \mathbb{R}$ and its opposite, are intervals in this sense.

\begin{lemma}\label{lem 18} 
Suppose that the locale $L$ is totally ordered.
Then the covering sieves for $a \in L$ are defined by the families of all $b$ such that $b < a$ or such that $b \leq a$. 
\end{lemma}

\begin{proof}
Suppose that a covering sieve $R \subset \hom(\ ,a)$ is generated by a set of elements $b_{i}$, so that $a = \vee_{i}\ b_{i}$. Suppose that $R \ne \hom(\ ,a)$.

Suppose that $c < a$. If $c$ is not bounded above by some $b_{i}$ then $b_{i} < c$ for all $i$ since $L$ is totally ordered, so that 
\begin{equation*}
a=\vee_{i}\ b_{i} \leq c < a,
\end{equation*}
and we have a contradiction. It follows that $c \leq b_{i}$ for some $i$, and so the relation $c < a$ is in $R$. 
\end{proof}

\begin{remark}
The collection of all $b$ such that $b \leq a$ is the trivial covering sieve for $a$, because it includes the identity relation on the object $a$. Lemma \ref{lem 18} says that an element $a \in L$ has at most two covering sieves if $L$ is totally ordered.

In order to be assured that $a \in L$ has a non-trivial covering, or that the elements $b < a$ cover $a$, we also need to know that $L$ satisfies condition 2) above, so that $L$ is an interval.
\end{remark}

\begin{example}
  The total ordering on $L$ is necessary for the conclusion of Lemma \ref{lem 18}.
  
The elements $(1,0)$ and $(0,1)$ define a covering of $(1,1)$ in $[0,1]^{\times 2}$, and the element $(\frac{1}{2},\frac{1}{2})$ is not bounded above by either $(1,0)$ or $(0,1)$.
\end{example}

{\it We shall assume that the locale $L$ is an interval for the rest of this section.}
\medskip

It follows from Lemma \ref{lem 18} that a presheaf $F$ on $L_{+}$ is a sheaf if and only if $F(0) = \ast$ and the map
\begin{equation*}
\eta: F(a) \to \varprojlim_{0< b < a}\ F(b) =: LF(a)
\end{equation*}
is an isomorphism for all $a \in L$ with $a$ not initial. There is no condition on $F(i)$ for the initial object $i$ of $L$ is initial, and $LF(i) = F(i)$. 

The assignment $a \mapsto LF(a)$ defines a presheaf $LF$ on $L_{+}$. Because $L$ has a total ordering and there are so few covering sieves for elements of $L$, the presheaf $LF$ is the universal separated presheaf associated to $F$ \cite[Lem 3.13]{LocHom}. 

In general, there is a canonical natural map $\eta: E \to LE$ for all presheaves $E$, and $LE$ is a sheaf if $E$ is separated. A presheaf $E$ is {\it separated} if the map $\eta: E \to LE$ is a sectionwise monomorphism. 

\begin{corollary}\label{cor 21}
If $F \in \Mon_{p}(L_{+})$, then the map $\eta$ is a sectionwise monomorphism, so that $F$ is a separated presheaf and $LF$ is its associated sheaf.
\end{corollary}

\begin{lemma}\label{lem 22}
If $F \in \Mon_{p}(L_{+})$, then $LF \in \Mon(L_{+})$. In particular, the associated sheaf functor 
\begin{equation*}
L=L^{2}: \Pre(L_{+}) \to \Shv(L_{+})
\end{equation*}
restricts to a functor $\Mon_{p}(L_{+}) \to \Mon(L_{+})$.
\end{lemma}

\begin{proof}
Suppose that $b \leq c$ in $L$. We show the restriction map
\begin{equation*}
LF(c) \to LF(b)
\end{equation*}
is a monomorphism.

Given compatible families $\{ x_{s} \}$ and $\{ y_{s} \}$ for $s < c$, if $x_{s} = y_{s}$ for $s < b$, then $x_{s}$ and $y_{s}$ have the same image in $F(t)$ for some $t<b$, and so $x_{s}=y_{s}$.
\end{proof}

\begin{example}
Suppose that $L=[0,1]$, let $A$ be a pointed set with base point $\ast$. Define a presheaf $F_{A}: (L_{+})^{op} \to \mathbf{Set}$ by
\begin{equation*}
F_{A}(s) =
\begin{cases}
\ast & \text{if $s=1$, and}\\
A & \text{if $s < 1$ in $[0,1]$}.
\end{cases}
\end{equation*}
Set $F_{A}(0) = \ast$, where $0$ is the new initial object of $[0,1]^{+}$.

If $s < 1$ the induced map $F_{A}(1) \to F_{A}(s)$ is the inclusion of the base point of $A$, and if $s \leq t < 1$ in $L$ then $F_{A}(t) \to F_{A}(s)$ is the identity on $A$. Then $\varprojlim_{s <1}\ F_{A}(s)$ is the set $A$ and not the base point $\ast$ in general, so that $F_{A}$ is a presheaf of monomorphisms, and is not a sheaf. 
\end{example}

\begin{example}
Suppose that $F_{i},\ i \in I$ is a list of objects in $\Mon_{p}(L_{+})$. Then the disjoint union $\sqcup_{i}\ F_{i}$ is in $\Mon_{p}(L_{+})$. Note that we must set $(\sqcup_{i}\ F_{i})(0) = \ast$ for this to work.
\end{example}

\begin{example}
Suppose that $A_{i} \subset F$ are subobjects of a fixed object $F \in \Mon_{p}(L_{+})$, so that all $A_{i}$ are in $\Mon_{p}(L_{+})$. Then the (sectionwise) union 
$\cup_{i}\ A_{i}$ is a subobject of $F$, and is also in $\Mon_{p}(L_{+})$.

It follows that the category $\Sub(F)$ of subobjects of an object $F \in \Mon_{p}(L_{+})$ is a locale.
\end{example}

Suppose that $E$ is a presheaf on $L_{+}$.
The epi-monic factorizations of the maps $E(s) \to E(i)$ for $s \in L$ determine subobjects $\Im(E)(s) \subset E(i)$ with commutative diagrams
\begin{equation*}
\xymatrix@R=8pt{
E(t) \ar[r] \ar[dd] & \Im(E)(t) \ar[dd] \ar[dr] \\
&& E(i) \\
E(s) \ar[r] & \Im(E)(s) \ar[ur]
}
\end{equation*}
for $s \leq t$. Set $\Im(E)(0) = \ast$. 

If $E$ is in $\Mon_{p}(L_{+})$, then the maps $E(t) \to \Im(E)(t)$ are isomorphisms. These constructions are functorial in presheaves $E$. 

We therefore have the following:

\begin{lemma}\label{lem 26}
There is a natural presheaf map $E \to \Im(E)$ such that $\Im(E)$ is in $\Mon_{p}(L_{+})$. This map is initial among all maps $E \to F$ with $F \in \Mon_{p}(L_{+})$, and so there is a natural bijection
\begin{equation*}
\hom_{\Mon_{p}(L_{+})}(\Im(E),F) \cong \hom(E,F),
\end{equation*}
so that the functor $E \mapsto \Im(E)$ is left adjoint to the inclusion of $\Mon_{p}(L_{+})$ in the presheaf category on $L_{+}$.
\end{lemma}

\begin{corollary}
  Suppose that $E$ is a presheaf on $L_{+}$.
Then the morphism  $E \to L(\Im(E))$ is initial among all presheaf maps $E \to F$ such that $F$ is in $\Mon(L_{+})$.
\end{corollary}

\begin{proof}
 The object $L(\Im(E))$ is the sheaf associated to $\Im(E)$ and it is a sheaf of monomorphisms by Lemma \ref{lem 22}. Use also the adjointness assertion of Lemma \ref{lem 26}.
  \end{proof}

Colimits of fuzzy sets can be described by the following result: 

\begin{lemma}\label{lem 28}
Suppose that $A: J \to \Mon(L_{+})$ is a small diagram in the category of sheaves with monomorphisms on $L_{+}$. Form the colimit 
\begin{equation*}
X = L(\Im(\varinjlim_{j \in J}\ A(j)))
\end{equation*}
in $\Mon(L_{+})$, and let $\psi_{X}: X(i) \to L$ be the corresponding fuzzy set.
Then
\begin{equation*}
\psi_{X}(x) = \vee_{j,y}\ \psi_{A(j)}(y),
\end{equation*}
where the index is over all pairs $j,y$ such that $y \mapsto x$ under a composite of the form
\begin{equation}\label{eq 3}
A(j)(s) \to \varinjlim_{j}\ A(j)(s) \to \Im(\varinjlim_{j}\ A(j))(s) \to X(i).
\end{equation}
\end{lemma}

\begin{proof}
We have
\begin{equation*}
X(i) = \varinjlim_{j}\ A(j)(i). 
\end{equation*}
and it follows that every $x \in X(i)$ is in the image of some composite (\ref{eq 3}).

Suppose that $y \in A(j)(s) \mapsto x$ under the composite (\ref{eq 3}). Then
\begin{equation*}
  \psi_{A(j)}(y) \leq \psi_{X}(x).
\end{equation*}
  This is true for all such pairs $(y,j)$ so that
\begin{equation*}
\vee_{j,y}\ \psi_{A(j)}(y) \leq \psi_{X}(x).
\end{equation*}

Suppose that $x \in X(i)$ lifts to $x' \in X(t)$, where $t$ is maximal. The element $x'$ is in the image of some composite
\begin{equation*}
A(i')(s) \to \varinjlim_{j}\ A(j)(t) \to \Im(\varinjlim_{j}\ A(j))(s) \to X(i)
\end{equation*}
for all $s<t$.
This means that there is an element $y' \in A(i')(t)$ which maps to $x'$ under the composite above, and so $s \leq \psi_{A(i')}(y')$ for all $s <t$. It follows that 
\begin{equation*}
t = \psi_{X}(x) \leq \vee_{i,y}\ \psi_{A(i)}(y).
\end{equation*}
\end{proof}

\begin{remark}
  Colimits of fuzzy sets and the left adjoint of the inclusion functor
  \begin{equation*}
    \Fuz(L)= \Mon(L_{+}) \subset \Shv(L_{+})
  \end{equation*}
  are described in Lemma 1.3 of Spivak's preprint \cite{fuzzy-Spivak}. In the present notation, that left adjoint is the functor $F \mapsto L^{2}\Im(F)$. The cocompleteness of the category of fuzzy sets follows from the cocompleteness of the sheaf category and the existence of the left adjoint of the inclusion.
  \end{remark}

\begin{example}
Form the union $A \cup B$ of two subsheaves $A,B \subset F$ of a sheaf $F \in \Mon(L_{+})$. Then there is a pushout diagram
\begin{equation*}
\xymatrix{
A \cap B \ar[r] \ar[d] & B \ar[d] \\
A \ar[r] & A \cup B
}
\end{equation*}
in $\Mon(L_{+})$. Here, $A \cup B$ is the sheaf $L(A \cup B)$ that is associated to the presheaf union $A \cup B$, which is in $\Mon_{p}(L_{+})$. Note that 
\begin{equation*}
A(i) \cup B(i) = (A \cup B)(i) = L(A \cup B)(i),
\end{equation*}
by the construction of $L(A \cup B)$. It follows from Lemma \ref{lem 28} that
\begin{equation*}
\psi_{A \cup B}(y) = \max\{ \psi_{A}(y), \psi_{B}(y) \}.
\end{equation*}
\end{example}

\section{Stalks}

Again, suppose that $L$ is a locale. Boolean localization theory involves a poset monomorphism $\omega: L \to B$ which takes values in a complete Boolean algebra, and induces a geometric morphism
\begin{equation*}
  \omega: \Shv(B) \to \Shv(L),
\end{equation*}
having  an inverse image functor $\omega^{\ast}: \Shv(L) \to \Shv(B)$ which is fully faithful.

The poset morphism $\omega: L \to B$ is relatively easy to describe, and the construction is reprised at the beginning of this section.

If $L$ is an interval, then the Boolean algebra $B$ can be identified with the power set $\mathcal{P}(L - \{1\})$ on the subset of $L$ that remains after removing the terminal object $1$. This leads immediately to a theory of stalks for sheaves and presheaves on $L$ (Example \ref{ex 31}, Lemma \ref{lem 32}, Lemma \ref{lem 34x}). These stalks have a simple construction in this case, and the expected behaviour of this theory can be verified directly (Lemma \ref{lem 32}).

This theory applies to the map $V(X) \to V(Y)$ of Vietoris-Rips systems that is associated to an inclusion $X \subset Y$ of data sets --- this is discussed in Example \ref{ex 34}. The present theory of stalks is used to show that the map $V(X) \to V(Y)$ is a local weak equivalence of simplicial sheaves on the locale $[0,R]^{op}$ ($R$ sufficiently large) if and only if $X=Y$ as data sets.

\medskip

Following \cite[p.51]{LocHom} and \cite{MM}, for $x \in L$, write
\begin{equation*}
  \neg x = \vee_{x \wedge y = 0}\ y.
\end{equation*}
The subobject $\neg\neg L$ of $L$ is defined to be the set of all $x \in L$ such that $\neg\neg x = x$. There is a frame morphism $\gamma: L \to \neg\neg L$ which is defined by $x \mapsto \neg\neg x$, since $\neg x = \neg\neg\neg x$ for all $x$. 

For $x \in L$, write $L_{\geq x}$ (as above) for the sublocale of objects $y$ with $y \geq x$. There is a homomorphism $\phi_{x}: L \to L_{\geq x}$ which is defined by $y \mapsto y \vee x$.

Let $\omega$ denote the composite frame morphism
\begin{equation}\label{eq 4}
L \xrightarrow{(\phi_{x})} \prod_{x \in L}\ L_{\geq x} \xrightarrow{(\gamma)}
\prod_{x \in L}\ \neg\neg L_{\geq x}.
\end{equation}
Then one knows (see, for example, \cite[p.52]{LocHom}) that $\omega$ is a monomorphism and that
\begin{equation*}
  B := \prod_{x\in L}\ \neg\neg L_{\geq x}
\end{equation*}
is a complete Boolean algebra.

Note that $L_{\geq 0} = L$ and that the morphism $\phi_{0}: L \to L_{\geq 0}$ is the identity.
%\medskip

The corresponding geometric morphism
\begin{equation*}
  \omega: \Shv(B) \to \Shv(L)
\end{equation*}
is a {\it Boolean localization} of $\Shv(L)$.
In particular, the inverse image functor
\begin{equation*}
  \omega^{\ast}: \Shv(L) \to \Shv(B)
\end{equation*}
is faithful, and is thus a {\it fat point} for the topos $\Shv(L)$.

The fat point assertion means that a map $E \to F$ of sheaves on $L$ is an monomorphism (respectively epimorphism, isomorphism) if and only if the induced map $\omega^{\ast}E \to \omega^{\ast}F$ is a monomorphism (respectively epimorphism, isomorphism) of sheaves on $B$. See \cite[Sec. 3.4]{LocHom} or \cite{J21}.

\begin{example}\label{ex 31}
Suppose that $L$ is totally ordered. If $x \in L$ and $x \ne 0$, then $x \wedge y = \min \{x,y\} = 0$ forces $y=0$. Thus,
\begin{equation*}
  \neg x =
  \begin{cases}
    0 & \text{if $x \ne 0$, and} \\
    1 & \text{if $x=0$.}
  \end{cases}
\end{equation*}

The corresponding Boolean algebra
\begin{equation*}
  B = \prod_{x \in L}\ \{x,1 \} \cong \prod_{x \in L-\{ 1 \}} \{x,1\}
\end{equation*}
is isomorphic to the power set $\mathcal{P}(L - \{ 1 \})$ of $L-\{1\}$, so that the sheaf category $\Shv(L)$ has enough points.

The poset map $\phi_{x}: L \to L_{\geq x}$ takes $y$ to $x$ if $y < x$ and takes $y$ to $y$ if $y \geq x$. It follows that the composite
\begin{equation*}
  L \xrightarrow{\phi_{x}} L_{\geq x} \xrightarrow{\gamma} \{x,1\}
\end{equation*}
takes $y$ to $x$ if $y \leq x$ and takes $y$ to $1$ if $y > x$.
The poset map $\omega: L \to \mathcal{P}(L-\{1\})$ therefore has the form
\begin{equation*}
  \omega(y) = L_{<y}
\end{equation*}
for $y \ne 0,1$.

For $x \in L-\{1\}$, the {\it stalk} $F_{x}$ of a sheaf $F$ on $L$ is defined by
\begin{equation}\label{eq 5}
  F_{x} = \varinjlim_{x < s}\ F(s).
\end{equation}
This colimit corresponds to the category of inclusions $\{ x \} \subset L_{<s}$, and so $F_{x}$ is the evaluation of the sheaf $\omega^{\ast}(F)$ at the set $\{ x \}$.
\medskip

The locale $L_{+}$ has a total order if $L$ has a total order.
For the object $0$, the stalk $F_{0}$ is isomorphic to the generic fibre:
\begin{equation*}
  F_{0} = \varinjlim_{s \in L}\ F(s) \cong F(i),
\end{equation*}
since $i$ is the initial object of $L$. The stalk
\begin{equation*}
F_{i} = \varinjlim_{i < s}\ F(s)
\end{equation*}
is more ``conventional''.
\end{example}

The following result is true by formal nonsense, given that we have a theory of stalks for sheaves on $L_{+}$ for a totally ordered locale $L$ in Example
\ref{ex 31}. The point of Lemma \ref{lem 32} (and its proof) is that it is easier to show directly that these stalks have the right properties in cases that one cares about for Data Science applications.

\begin{lemma}\label{lem 32}
Suppose that the locale $L$ is an interval and that the map $E \to F$ is a stalkwise epimorphism (respectively stalkwise monomorphism, stalkwise isomorphism) of sheaves of monomorphisms on $L_{+}$. Then the map $E \to F$ is an epimorphism (respectively monomorphism, isomorphism) of sheaves.
\end{lemma}

One says that a map $E \to F$ of sheaves on $L_{+}$ is a {\it stalkwise} epimorphism if the induced functions $E_{x} \to F_{x}$ are surjective for all $x \in L - \{1 \}$. Similarly, stalkwise monomorphisms (respectively stalkwise isomorphisms) are defined by the requirement that all induced functions in stalks are injective (respectively bijective).

\begin{proof}[Proof of Lemma \ref{lem 32}]
Suppose that the map $E \to F$ is a stalkwise epimorphism.

By the observation in Example \ref{ex 31}, there is a natural isomorphism $F_{0} \cong F(i)$ for all sheaves $F$ on $L_{+}$, so that the map $E(i) \to F(i)$ is an epimorphism.

Suppose that $x \in L - \{1\}$. Then the collection of relations $t < x$ is a covering family. Take $u \in F(x)$ and let $u_{t} \in F(t)$ be its image under the restriction map $F(x) \to F(t)$, where $t < x$. Then $u_{t}$ represents an element of $F_{t}$, and the map $E_{t} \to F_{t}$ is an epimorphism. It follows that there is an element $v_{t} \in E(t)$ such that $v_{t} \mapsto u_{t} \in F(t)$. This means that the sheaf map $E \to F$ is a local epimorphism \cite[Sec. 3.2]{LocHom}, and it is therefore an epimorphism of sheaves.

If $E \to F$ is a stalkwise monomorphism, then the map $E(i) \to F(i)$ is a monomorphism. For each $x \in L$ there is a commutative diagram
\begin{equation*}
  \xymatrix{
    E(x) \ar[r] \ar[d] & E(i) \ar[d] \\
    F(x) \ar[r] & F(i)
  }
\end{equation*}
The horizontal morphisms are monomorphisms since $E$ and $F$ are sheaves of monomorphisms, so the map $E(x) \to F(x)$ is a monomorphism. This is true for all $x \in L$, and so the sheaf map $E \to F$ is a monomorphism.

If the map $E \to F$ is a stalkwise isomorphism, then it is both a stalkwise epimorphism and a stalkwise monomorphism, so that $E \to F$ is an epimorphism and a monomorphism of sheaves by the previous paragraphs. It follows that $E \to F$ is an isomorphism of sheaves.
\end{proof}

\begin{example}\label{ex 33}
  Suppose that the locale
  \begin{equation*}
    L= L_{1} \times \dots \times L_{k}
  \end{equation*}
  is a product of intervals $L_{i}$. 

  The construction of the locale morphism (\ref{eq 4}) preserves products, so that the sheaf category $\Shv(L)$ again has enough points.

  The poset map $\omega$ has the form
  \begin{equation*}
    \begin{aligned}
      L = L_{1} \times \dots \times L_{k} \xrightarrow{\omega \times \dots \times \omega} &\mathcal{P}(L_{1}-\{1 \}) \times \dots \times \mathcal{P}(L_{k}-\{1\})\\
      &=
    \mathcal{P}((L_{1}-\{1\}) \sqcup \dots \sqcup (L_{k}-\{1\})),
    \end{aligned}
  \end{equation*}
  and takes $(y_{1}, \dots, y_{k})$ to the disjoint union $(L_{1})_{<y_{1}} \sqcup \dots \sqcup (L_{k})_{<y_{k}}$.

  If $x \in L_{1}$ and $F$ is a sheaf on $L=L_{1} \times \dots \times L_{k}$, then
  \begin{equation*}
    F_{x} = \varinjlim_{x<s}\ F(s,0, \dots ,0).
  \end{equation*}
  In effect, the collection of all $k$-tuples $(s,0,\dots ,0)$ with $s>x$ is cofinal in the collection of all $k$-tuples $(s_{1}, \dots ,s_{k})$ with $s_{1}>x$.

  It follows that $F_{x}$ is the stalk at $x$ of the restriction of $F$ along the poset morphism
  \begin{equation*}
    i_{1}: L_{1} \to L_{1} \times \dots \times L_{k}
  \end{equation*}
  which is defined by $s \mapsto (s,0, \dots ,0)$.
\end{example}

We now show that the category of presheaves of sets on an interval $L$ has a theory of stalks that specializes to the definition of stalks for sheaves on $L$ that we have from Example \ref{ex 31} and Lemma \ref{lem 32}, and such that the associated sheaf map $\eta: F \to L^{2}(F)$ is a stalkwise isomorphism for all presheaves $F$.

The theory of stalks for presheaves on and interval is analogous to the theory of stalks for presheaves on a topological space, and the theory of stalks for presheaves on the \'etale site of a scheme.

\begin{lemma}\label{lem 34x}
Suppose that the locale $L$ is an interval, and let $F$ be a presheaf on $L$. Define
\begin{equation*}
F_{x} = \varinjlim_{x<s}\ F(s)
\end{equation*}
for all $x \in L - \{1\}$. Then we have the following:
\begin{itemize}
\item[1)] The set $F_{x}$ is the stalk at $x$ as in (\ref{eq 5}) if $F$ is a sheaf on $L$.
\item[2)] The associated sheaf map $\eta: F \to L^{2}F$ induces bijections $F_{x} \xrightarrow{\cong} L^{2}F_{x}$ for all $x \in L-\{1\}$.
\item[3)] A map $E \to F$ of presheaves induces bijections $E_{x} \xrightarrow{\cong} F_{x}$ for all $x \in L-\{1\}$ if and only if the map $L^{2}E \to L^{2}F$ of associated sheaves is an isomorphism.
\end{itemize}
\end{lemma}

\begin{proof}
Write
  $I = L - \{1\}$
so that the  poset morphism $\omega$ of Example \ref{ex 31} has the form
$\omega: L \to \mathcal{P}(I)$.

The direct image (restriction) functors $\omega_{\ast}$ and inclusion functors $i$ for the various presheaf and sheaf categories fit into a commutative diagram
\begin{equation*}
  \xymatrix{
    \Pre(\mathcal{P}(I)) \ar[r]^{\omega_{\ast}} & \Pre(L) \\
    \Shv(\mathcal{P}(I)) \ar[u]^{i} \ar[r]_{\omega_{\ast}} & \Shv(L) \ar[u]_{i}
  }
  \end{equation*}
There is a natural isomorphism of left adjoint functors
\begin{equation*}
  L^{2}\omega^{p}(F) \xrightarrow{\cong} L^{2}\omega^{p}L^{2}(F) = \omega^{\ast}L^{2}(F)
\end{equation*}
that is induced by the associated sheaf map $\eta: F \to L^{2}(F)$.
Here, $\omega^{p}$ is the left Kan extension of the restriction functor $\omega_{\ast}$ on the presheaf level.

By definition of the left Kan extension of $\omega_{\ast}$, we have an isomorphism
\begin{equation*}
  \omega^{p}F(\{x \}) \cong \varinjlim_{x < s}\ F(s)
\end{equation*}
for all $x \in L-\{1\}$.

It follows that $F_{x} = \omega^{p}F(\{x\})$ for all presheaves $F$ and $x \in I=L-\{1\}$. Observe also that $F_{x}$ is the stalk of $F$ at $x$ which is defined in (\ref{eq 5}) if $F$ is a sheaf, so that statement 1) holds.

For $x \in I$ and a presheaf $E$ on the power set $\mathcal{P}(I)$, the associated sheaf map induces a bijection
\begin{equation*}
  E(\{x\}) \xrightarrow{\cong} L^{2}E(\{ x\}).
\end{equation*}
It follows that the functions
\begin{equation*}
F_{x} = \omega^{p}F(\{x \}) \xrightarrow{\omega^{p}\eta} \omega^{p}L^{2}F(\{x\}) = (L^{2}F)_{x}
\end{equation*}
are bijections for all presheaves $F$ on $L$ and $x \in L-\{1\}$, giving statement 2).

Statement 3) follows from Lemma \ref{lem 32}. Alternatively,
statement 3) is a consequence of statements 1) and 2), and the fact that the inverse image functor $\omega^{\ast}: \Shv(L) \to \Shv(\mathcal{P}(I))$ is fully faithful. 
\end{proof}

  \begin{example}\label{ex 34}
  Suppose, as in Example \ref{ex 10}, that $X \subset \mathbb{R}^{n}$ is a data set, with ordering $X \subset \mathbf{N}$, and choose $R > d(x,y)$ for all pairs of points $x,y \in X$.

  The association $s \mapsto V_{s}(X)$ for $s \in [0,R]$ defines a simplicial sheaf $V(X)$ (of Vietoris-Rips complexes) of monomorphisms on the totally ordered locale $[0,R]^{op}_{+}$. The stalk $V(X)_{t}$ for $t \in (0,R]$ is defined by
  \begin{equation*}
    V(X)_{t} = \varinjlim_{s < t}\ V_{s}(X),
  \end{equation*}
  where the indicated ordering is that of $[0,R]$.

  Note that
\begin{equation*}
  V(X)_{i} = \varinjlim_{s < R}\ V_{s}(X) = \Delta^{X},
\end{equation*}
because
we have chosen $R > d(x,y)$ for all pairs of points $x,y \in X$. 

Observe as well that for small numbers $t$, the stalk $V(X)_{t}$ is the discrete space on the set $X$.

Suppose that $X \subset Y \subset \mathbb{R}^{n}$ are data sets and $R > d(x,y)$ for all pairs of points $x,y \in Y$ (hence in $X$). Then the inclusion $X \subset Y$ defines a map of simplicial sheaves (of monomorphisms) $V(X) \to V(Y)$. This map is a local weak equivalence if and only if $X=Y$, because $V(X)_{t}$ and $V(Y)_{t}$ are discrete for small numbers $t$.
\end{example}

%\vfill\eject

\bibliographystyle{plainnat}
\bibliography{spt}

\end{document}